\documentclass{my-jgaa-art}
\usepackage{amsmath}
\usepackage{amssymb}
\usepackage[disableredefinitions]{complexity}
\usepackage{hyperref}
\usepackage{cite}
\usepackage[capitalize]{cleveref}
\usepackage{subcaption}
\usepackage{thm-restate}
\usepackage{todonotes}
\usepackage[T1]{fontenc}
\usepackage{etoolbox}

\DeclareMathOperator{\width}{width}

\DeclareMathOperator{\Max}{Max}
\DeclareMathOperator{\dmap}{dmap}
\DeclareMathOperator{\at}{at}
\DeclareMathOperator{\suc}{succ}
\DeclareMathOperator{\lva}{lva}

\definecolor{dark blue}{rgb}{0.121,0.47,0.705}

\let\emph\relax\DeclareTextFontCommand{\emph}{\color{dark blue}\em}

\newtheorem{theorem}{Theorem}
\newtheorem{lemma}[theorem]{Lemma}
\newtheorem{corollary}[theorem]{Corollary}
\newtheorem{proposition}[theorem]{Proposition}

\newcommand{\restateref}[1]{\ifnumless{\value{section}}{2}{\hyperref[#1*]{$\blacktriangledown$}}{\hyperref[#1]{$\blacktriangle$}}}

\begin{document}
\HeadingAuthor{S.~Felsner, J.~Hodor, G.~Ortali, and A.~Wolff}
\HeadingTitle{Navigating Posets with Few Maps} %
\title{\renewcommand*{\thefootnote}{\fnsymbol{footnote}}Navigating
  Posets with Few Maps\footnote{This work was started at the workshop
    Homonolo 2024 in Nov\'a Louka.  We thank the organizers and the
    other participants; in particular Bartosz Podkanowicz and Felix
    Schr\"oder.  We also thank Oscar Fontaine for raising the question
    whether the atlas thickness can be approximated.  J.~Hodor is
    supported by the National Science Center of Poland under grant
    UMO-2022/47/B/ST6/02837 within the OPUS 24 program.  A preliminary
    version of this work has appeared in Proc.\ 34th International
    Symposium on Graph Drawing \& Network Visualization (GD 2026).}}

\authorOrcid[first]{Stefan Felsner}{felsner@math.tu-berlin.de}{0000-0002-6150-1998}
\authorOrcid[second]{J\k{e}drzej Hodor}{jedrzej.hodor@gmail.com}{0000-0002-2564-7121}
\authorOrcid[third]{Giacomo Ortali}{giacomo.ortali@uninettunouniversity.net}{0000-0002-4481-698X}
\authorOrcid[fourth]{Alexander Wolff}{}{0000-0001-5872-718X}

\affiliation[first]{Technische Universit\"at Berlin, Germany}

\affiliation[second]{Theoretical Computer Science Department, 
Faculty of Mathematics and Computer Science and  Doctoral School of Exact and Natural Sciences, Jagiellonian University, Krak\'ow, Poland}

\affiliation[third]{Università Telematica Internazionale Uninettuno, Rome, Italy}

\affiliation[fourth]{Universit\"at W\"urzburg, Germany}

\maketitle

\begin{abstract}
  We study two new parameters for finite posets motivated by the
  problem of efficiently determining the set of successors of a given
  element.  A \emph{plane map} of a poset~$P=(X,\leq)$ is an injective
  mapping of~$X$ into the Cartesian plane $\mathbb{R}^2$.  Given two
  different points~$a$ and~$b$ in the plane, we say that $b$
  \emph{dominates}~$a$ if $a<b$ coordinatewise.  We say that an
  element $x$ of $P$ is \emph{tight} in a plane map~$\mu$ if the
  following holds: $x<y$ in $P$ if and only if $\mu(y)$ dominates
  $\mu(x)$.  Note that, by definition, every 2-dimensional poset
  admits a map such that every element of the poset is tight.  For any
  poset~$P$, we define the \emph{mapability} of~$P$, $\dmap(P)$, to be
  the maximum number of elements that are tight in a single map, and
  we define the \emph{atlas thickness} of~$P$, $\at(P)$, to be the
  size of the smallest collection of maps such that every element is
  tight in at least one map of the collection.

  We relate these parameters to the classical notions of dimension and
  width: for every poset~$P$, we show that
  $\dim(P) \le 2\at(P) \le \width(P)+1$.  On the other hand, there
  exists a sequence of posets~$(P_n)_{n \ge 1}$ such that the atlas
  thickness of $P_n$ is doubly exponential in the dimension of~$P_n$.

  On the computational side, we prove that it is NP-complete, for a
  given poset~$P$, to compute the mapability of~$P$ and to decide
  whether $\at(P) \le 2$.  In contrast to the latter, we show that
  computing the mapability of a poset is fixed-parameter tractable
  with respect to the natural parameter.
\end{abstract}

\section{Introduction}

In combinatorics and in computer science there is interest in compact
encodings of posets such that the presence of an order relation
between two elements can be tested efficiently.  We propose and study
a new idea for such an encoding. The encoding consists of an atlas of
pages which show the elements of the poset on plane maps. We aim for a
small number of pages such that, for each element~$x$, there is a page
which shows all relations $x \leq y$.

We assume familiarity with basic notions of the theory of
\emph{partially ordered sets} (posets) as used, e.g., in West's book
{\em Combinatorial Mathematics}~\cite{W21}.  Due to their relevance to
our topic, we nevertheless recall the definition of dimension and some
related concepts.  A \emph{realizer} of a poset~$P$ is a collection
$\cal R$ of linear extensions of $P$ such that, for every incomparable
pair of elements~$x$ and~$y$ in~$P$, there are $L$ and $L'$ in
$\cal R$ such that $x <_L y$ and $y <_{L'} x$.  This is equivalent to
saying that $P$ is the intersection of the linear extensions
in~$\cal R$.  The \emph{dimension} of a poset~$P$, denoted
\emph{$\dim(P)$}, is the minimum cardinality of a realizer of~$P$.  It
is well known that $\dim(P)$ also equals the minimum $k$ such that $P$
is a suborder of $\mathbb{R}^k$ with the \emph{product order}, also
known as \emph{dominance order $\le_\mathrm{dom}$}, where
$x \le_\mathrm{dom} y$ if and only if $x_i \leq y_i$ for each
$i \in \{1,\ldots,k\}$.  In a poset~$P$, a \emph{chain} is a subset
of~$P$ of pairwise comparable elements (that is, they are linearly
ordered), whereas an \emph{antichain} is a subset of~$P$ of pairwise
incomparable elements.  The \emph{width} of a poset~$P$, $\width(P)$,
is the size of a largest antichain in~$P$.  The \emph{size} of a
poset~$P$, \emph{$|P|$}, is the number of elements that~$P$ contains.%
The \emph{incidence poset} of a graph $G$, denoted by~\emph{$P_G$}, is
the poset with ground set $V(G) \cup E(G)$ such that, for every two
elements~$x$ and~$y$ of~$P_G$, we have $x < y$ if and only if
$x \in V(G)$, $y \in E(G)$, and $x$ is incident to~$y$ in~$G$.
For an example, see \cref{fig:hard-1}.

\begin{figure}[h]
\centering
    \begin{subfigure}{.23\textwidth}
        \centering
        \includegraphics[page=1]{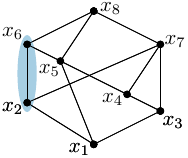}
        \subcaption{}
        \label{fig:hard-1-a}
    \end{subfigure}
    \hfil
    \begin{subfigure}{.51\textwidth}
        \centering
        \includegraphics[page=3]{drawdim-hard}
        \subcaption{}
        \label{fig:hard-1-b}
    \end{subfigure}
    \caption{(a) A graph $G$ and (b)~the incidence poset~$P_G$ of~$G$,
      represented as a DAG.}
    \label{fig:hard-1}
\end{figure} 

A common motivation for studying the dimension of posets is that a
realizer of size~$k$ for a poset of size~$n$ provides a data structure
of size $\tilde{\mathcal{O}}(nk)$ that allows to determine the
comparability status of a pair $(x,y)$ with $k$ lookup operations. To
improve the size of the data structure, Ne\v{s}et\v{r}il and
Pudl\'ak~\cite{NP89} introduced the concept of \emph{Boolean
  dimension}. This provides a data structure that in general is of
smaller size and requires fewer lookup operations but may need a long
computation to determine the comparability status of a pair.

In this paper, we go the other direction.  We are willing to use more
storage but for a given element $x$ of a poset $P$, we want to find
the set \emph{$\suc_P[x]$} $=\{ y \in P : x \leq_P y \}$ quickly.

Here is a first approach.  Let $P$ be a poset, and let $\cal T$ be a
set of permutations of its elements such that, for every~$x$ in~$P$,
there is a $\pi$ in~$\cal T$ such that
$\suc_P[x] = \{ y \in P : x \leq_\pi y \}$. In this case we say that
$x$ is \emph{tight} in $\pi$. How many permutations are needed such
that all elements of $P$ are tight in at least one of them?  The
question has a simple answer. Observe that if $x$ and $y$ are
incomparable, then at most one of $x$ and $y$ can be tight in a single
permutation. Hence, the tight elements of a permutation form a chain
and chains of tight elements of all permutations in $\cal T$ form a
chain cover of $P$.  From Dilworth's theorem \cite[p.~546]{W21}, we
therefore get the lower bound $|{\cal T}| \geq \width(P)$.  Since, for
every chain~$C$ of~$P$, there is a linear extension~$L$ such that all
elements of~$C$ are tight in~$L$, we see that $\width(P)$ is the
minimum size of a family $\cal T$ and that indeed it is enough to use
linear extensions in the family. A drawback of this approach is that
it would use $n$ linear extensions just to get all elements of a
size-$n$ antichain tight.

This motivates our second approach.  A \emph{(plane) map} for a
poset~$P$ is an injective mapping~$\mu$ of the elements of~$P$ into
the Cartesian plane~$\mathbb{R}^2$.  We use plane maps because planar
point sets are easier to grasp for humans than point sets in
higher-dimensional spaces.

An element~$x$ of~$P$ is \emph{tight} on a map $\mu$ if
$\suc_P[x] = \{ y \in P : \mu(x) \le_\mathrm{dom} \mu(y) \}$.  
For example, the element $b$ of the poset~$Q$ shown in
\cref{fig:intro-a} is tight on the map in \cref{fig:intro-c} but not
on the map in \cref{fig:intro-b}.
Note that there exists a mapping of the elements of a poset $P$ into
$\mathbb{R}^k$ such that every element is tight if and only if the
dimension of $P$ is at most $k$.
We are interested in the following two parameters of a poset~$P$:
\begin{itemize}
\item The \emph{(dominance) mapability} of $P$, denoted
  \emph{$\dmap(P)$}, is the maximum number of tight elements on a
  single plane map for~$P$.
\item Let an \emph{atlas} for~$P$ be a collection of plane maps
  for~$P$ such that every element of~$P$ is tight on at least one of
  the maps.  Then the \emph{atlas thickness} of $P$, denoted
  \emph{$\at(P)$}, is the minimum number of maps that together form an
  atlas for~$P$.
\end{itemize}
We refer to \cref{fig:intro}.  \cref{fig:intro-a} depicts a poset~$Q$.
\cref{fig:intro-b,fig:intro-c} show a collection of two maps where
every element of $Q$ is tight on at least one map, since elements
$\{a,c,ab,ac,bc\}$ are tight in the former and $b$ is tight in the
latter. As we are going to observe later, $Q$ is the crown poset,
which has dimension~$3$.  Hence, it is not possible to find a map
where more than five elements of~$Q$ are tight, or a collection of
maps for~$Q$ with fewer than two maps.  This implies that $Q$ has
mapability~$5$ and atlas thickness~$2$.

Note that, by definition, every 2-dimensional poset admits a single
map such that every element of the poset is tight.

\begin{figure}[tb]
\centering
    \begin{subfigure}{.21\textwidth}
        \centering
        \includegraphics[page=5]{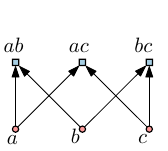}
        \subcaption{}
        \label{fig:intro-a}
    \end{subfigure}
    \hfil
    \begin{subfigure}{.22\textwidth}
        \centering
        \includegraphics[page=6]{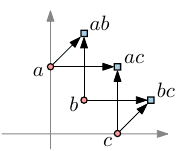}
        \subcaption{}
        \label{fig:intro-b}
    \end{subfigure}
    \hfil
    \begin{subfigure}{.21\textwidth}
        \centering
        \includegraphics[page=7]{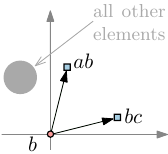}
        \subcaption{}
        \label{fig:intro-c}
    \end{subfigure}
    \caption{(a) The so-called crown poset $Q$. (b)~A map of $Q$ where
      all elements but $b$ are tight. (c)~A map of $Q$ where $b$ is
      tight.}
    \label{fig:intro}
\end{figure} 

Now that we have introduced the notion of an atlas, we briefly return
to our motivation.  Given an atlas~$A$ for~$P$ and an element~$x$
of~$P$, we can use a simple look-up table to determine the map in~$A$
on which~$x$ is tight.  If, for each map~$\mu$ in~$A$, we store the
points in $\mu(P)$ in a separate priority search tree~$T_\mu$, we can
query~$T_\mu$ with the quadrant
$[\mu(x)_1,\infty]\times[\mu(x)_2,\infty]$ (see paragraph ``Notation''
on page~\pageref{notation}) to determine the set $\suc_P[x]$ in
$O(|\suc_P[x]|+\log |P|)$ time~\cite{McCreight-pst-SICOMP85}.  The
total space consumption of the priority search trees is
$O(\at(P) \cdot |P|)$; the total preprocessing time for setting up the
trees is their space consumption times a factor of $O(\log |P|)$.

While our definition requires each map in~$A$ to define locations for
every element in~$P$, this is not needed if we just want to solve the
above query problem.  For example, in the map depicted in
\cref{fig:intro-c}, there is no need to store the points in the gray
circle.  Hence, the total space consumption is just
$O(\sum_{x \in P} |\suc_P[x]|)$, which is the number of pairs of
comparable elements in~$P$.
  
\paragraph{Our contribution.}

We relate the atlas thickness of a poset to its dimension and width;
see \Cref{sec:comparison}.  (Click on the triangle to jump directly to
the proof.)

\begin{restatable}[\restateref{thm:dim-le-2pw}]{theorem}{thmdimlepw}
  \label{thm:dim-le-2pw}
  For every poset $P$,
  $\dim(P)\le 2\at(P) \leq \width(P)+(\width(P) \bmod 2)$.
\end{restatable}

The inequalities in~\Cref{thm:dim-le-2pw} are tight, however, there
are posets whose atlas thickness is much larger than their dimension.
Note also that antichains have atlas thickness~$1$ and arbitrary
width.  The next two statements are proved in~\Cref{sec:lower}.
\Cref{thm:dim-eq-2at} is witnessed by the family of standard examples,
and~\Cref{thm:pw-2-2-dim} is witnessed by the family of the incidence
posets of complete graphs.

\begin{restatable}[\restateref{thm:dim-eq-2at}]{theorem}{thmdimeqat}
  \label{thm:dim-eq-2at}
  For every positive integer $n$, there is a poset $P_{n}$ with
  $ \dim(P_{n}) = 2\at(P_{n})=2n$.
\end{restatable}

\begin{restatable}[\restateref{thm:pw-2-2-dim}]{theorem}{thmpwgedim}
  \label{thm:pw-2-2-dim}
  There exists a sequence of posets $(P_n)_{n \ge 1}$ such that
  $\at(P_n) = 2^{2^{\Omega(\dim(P_n))}}$.
\end{restatable}

The problem of checking whether a given poset has dimension at most
$k$ is \NP-complete already for $k = 3$~\cite{Y82}.  Moreover, there
exists no polynomial-time algorithm that approximates the dimension of
a given poset~$P$ within a factor of $O(|P|^{1-\varepsilon})$ for
any $\varepsilon>0$, unless $\NP=\ZPP$~\cite{cln-gpr-SODA13}.

In order to analyze the complexity of computing our new poset parameters,
we first establish a (somewhat surprising) connection to a graph
parameter.  The \emph{linear vertex arboricity} of a graph~$G$,
\emph{$\lva(G)$}, is the smallest number~$r$ such that the vertex set
of $G$ can be partitioned into $r$ subsets each of which induces a
linear forest.

\begin{restatable}[\restateref{thm:lva}]{theorem}{thmlva}
  \label{thm:lva}
  For every graph~$G$, the linear vertex arboricity of~$G$ equals the
  atlas thickness of the incidence poset of~$G$, that is,
  $\lva(G)=\at(P_G)$.%
\end{restatable}

Since it is \NP-hard to test whether a given graph has linear vertex
arboricity at most~2~\cite[Theorem~12]{cflrvw-cdgfl-JGAA23}, the same
holds for the atlas thickness of a poset.  Since the linear vertex
arboricity is linearly linked to the chromatic
number~\cite[Theorem~12]{cflrvw-cdgfl-JGAA23}, we even obtain
non-approximability.  (The authors of \cite{cflrvw-cdgfl-JGAA23} were
interested in a geometric graph parameter, $\pi^1_3$, which also
coincides with the linear vertex
arboricity~\cite{cflrvw-dgflf-JoCG20}.)

\begin{restatable}[\restateref{thm:hardness-at}]{corollary}{corhardnessat}
  \label{thm:hardness-at}
  It is \NP-complete to test whether the atlas thickness of a poset is
  at most~$2$.
  Approximating the atlas thickness of a poset~$P$ within a factor of
  $O(|P|^{1/2-\varepsilon})$ is \NP-hard.%
\end{restatable}

\begin{restatable}[\restateref{thm:hardness-dmap}]{theorem}{thmhardnessdmap}
  \label{thm:hardness-dmap}
  It is \NP-complete to decide, given a poset $P$ and a positive
  integer~$k$, whether $\dmap(P) \le k$.  Hence, computing the
  mapability of a poset is \NP-hard.
\end{restatable}

\Cref{thm:hardness-at} shows that computing the atlas thickness of a
poset is \textsf{para}\NP-hard with respect to the natural parameter.
This is in contrast to the mapability of a poset, which turns out to
be fixed-parameter tractable (\FPT) with respect to the natural
parameter; see \Cref{se:FPT}.  We prove that the size of each poset is
bounded in terms of its mapability.

\begin{restatable}[\restateref{thm:size-dmap}]{theorem}{thmsizetodmap}
  \label{thm:size-dmap}
  There exists a computable function $f$ such that, for every
  poset~$P$, we have $|P| \le f(\dmap(P))$.
\end{restatable}

This directly yields an \FPT-algorithm with respect to the mapability
of the given poset.

\begin{corollary}
  \label{cor:FPT}
  There exists a computable function $g$ and an algorithm that, given
  a poset~$P$, returns the value $\dmap(P)$ in time $g(\dmap(P))$.
\end{corollary}

\paragraph{Related work.}

Our work follows a recent trend in graph drawing to visualize a
complex combinatorial or geometric object (say, a graph) by means of a
collection of drawings, each of which includes all parts (say, edges),
but highlights only a subset of them.  For example, Hlin\v{e}n{\'{y}}
and Masa\v{r}{\'{\i}}k~\cite{hm-mucd-GD23} introduced the notion of an
\emph{uncrossed collection} of a graph, that is, a collection of
straight-line drawings of the graph such that every edge of the graph
is uncrossed in at least one drawing in the collection.  They
considered two optimization problems; namely minimizing the size of
the collection (the uncrossed number) and minimizing the total number
of crossings in the collection (the uncrossed crossing number).  They
showed the latter problem \NP-hard but in \FPT\ with respect to the
natural parameter.
Similarly, Anti{\'c}, Liotta, Masa\v{r}\'{\i}k, Ortali, Pfretzschner,
Stumpf, Wolff, and Zink~\cite{almopswz-ucod-WG25} introduced
\emph{unbent collections} for plane 4-graphs, that is, for embedded
planar graphs of vertex degree at most~4.  Such a collection contains
orthogonal drawings of the given graph such that every edge of the
graph is straight in at least one drawing in the collection.  They
also considered the corresponding optimization problems of computing,
for a given plane 4-graph, the unbent number (which they proved to be
always at most~3) and the unbent bend number (for which they showed
\NP-hardness and gave a 3-approximation algorithm).

Less related are
storyplans~\cite{BinucciGLLMNS22,BinucciGLLMNS24,FialaFLWZ24} and
graph stories~\cite{BattistaDGGOPT22,BattistaDGGOPT23}, where the
drawings in the collection show only parts of the graph (with a
prescribed property such as planarity) and the order of the drawings
is relevant.

\paragraph{Notation.}
\label{notation}

For a positive integer $k$, we use $[k]$ as shorthand for
$\{1,2,\dots,k\}$.  For a poset~$P$, we let \emph{$\Max(P)$} denote
the set of all maximal elements of~$P$.
For a map $\mu$ and an element $x$ of
a poset $P$, we write \emph{$\mu(x)_1$} and \emph{$\mu(x)_2$} for the
first and the second coordinates of $\mu(x) \in \mathbb{R}^2$,
respectively.

For a poset $P$, for each element $x$ of $P$, we write
$\suc_P[x] =\{ y \in P : x \leq_P y \}$ and
$\suc_P(x) =\{ y \in P : x <_P y \}$ to denote, respectively, the
\emph{closed} and the \emph{open} set of \emph{successors} of $x$.

The pieces of text in gray explain statements that may be obvious for
readers who feel at home with posets.

\section{Comparison Between Parameters}
\label{sec:comparison}

The following lemma, which may be of independent interest, is the main
tool for the proof of \cref{thm:dim-le-2pw}.

\begin{lemma}
  \label{lem:extensions-on-coordinates}
  For every poset $P$ and for every map $\mu$ for $P$, there exists a
  map $\mu'$ for $P$ such that each element of $P$ that is tight on
  $\mu$ is tight on $\mu'$, and the two linear orders of the elements
  of $P$ given by the orders on the coordinates of $\mu'$ are linear
  extensions of $P$.
\end{lemma}
\begin{proof}
  Let $T = \{t_1,\dots,t_k\}$ be the set of elements of $P$ that are
  tight on $\mu$.  Note that some elements of $P$ may have the same
  value under $\mu(\cdot)_1$ or $\mu(\cdot)_2$.  Let $x_1,\dots,x_k$
  be any linear order on the elements of $T$ extending the order given
  by $\mu(\cdot)_1$, and let $y_1,\dots,y_k$ be any linear order on
  the elements of $T$ extending $\mu(\cdot)_2$.  For each element $p$
  of $P$ with $p \notin T$, let
  \[i(p) = \max(\{i \in [k] : x_i < p\} \cup \{0\} )\ \ \text{ and } \
    \ j(p) = \max(\{j \in [k] : y_j < p\} \cup \{0\}).\] In other
  words, $x_{i(p)}$ is the largest element of $T$ with respect to
  $\mu(\cdot)_1$ that is below $p$ in $P$, and $y_{j(p)}$ is defined
  analogously.  Next, for each $i \in [k] \cup\{0\}$, let
  \[X_i = \{p \in P\setminus T : i(p) = i\} \ \ \text{ and } \ \ Y_i =
    \{p \in P\setminus T : j(p) = i\}.\] %
  For each $i \in [k] \cup \{0\}$, let $K_i$ be an arbitrary linear
  extension of the subposet of $P$ induced by the elements of $X_i$,
  and let $L_i$ be an arbitrary linear extension of the subposet of
  $P$ induced by the elements of $Y_i$.  Finally, let $\mu'$ be
  defined so that the linear order given by $\mu'(\cdot)_1$ is the
  concatenation $K = K_0x_1K_1x_2K_2\dots K_{k-1}x_kK_k$, and the
  linear order given by $\mu'(\cdot)_2$ is the concatenation
  $L = L_0y_1L_1y_2L_2\dots L_{k-1}y_kL_k$.
  
  To complete the proof, we need to verify that the elements of $T$
  are tight on $\mu'$, and that $L$ and $K$ are linear extensions of
  $P$. An element $p\not\in T$ can be pushed left and down.  By
  definition of $i(p)$ and $j(p)$, however, it will not pass any
  $t\in T$ with $t < p$ in $P$. Hence, elements of $T$ remain
  dominated by all their successors in $P$ which means they remain
  tight.
    
  \textcolor{gray}{More formally, let $t \in T$ be such that
    $t = x_i = y_j$.  We will show that
    $\suc_P[t] = \{p \in P : \mu'(t) \leq \mu'(p)\}$.  First, we have
    $\{p \in P : \mu(t) \leq \mu(p)\} \cap T = \{p \in P : \mu'(t)
    \leq \mu'(p)\} \cap T$.  Next, let $p$ be an element of $P$ with
    $p \notin T$.  If $p \in \suc_P(t)$, i.e., $t < p$ in $P$, then
    $i \leq i(p)$ and $j \leq j(p)$, hence, $t < p$ in both $L$ and
    $K$, and so, $\mu'(t) \leq \mu'(p)$.  If $\mu'(t) \leq \mu'(p)$,
    then $i \leq i(p)$ and $j \leq j(p)$, hence, since the elements of
    $T$ are tight on $\mu$, we have
    $\mu(t)_1 \leq \mu(x_{i(p)})_1 < \mu(p)_1$ and
    $\mu(t)_2 \leq \mu(y_{j(p)})_2 < \mu(p)_2$.  It follows that
    $t \leq p$ in $P$, as desired.}

  Note that $K$ and $L$ are linear extensions of $P$.
  \textcolor{gray}{Indeed, let $p$ and $p'$ be distinct elements
    of~$P$ with $p < p'$, we claim that $p < p'$ in $K$.  If
    $p,p' \in T$, then this is clear since the elements of $T$ are
    tight on $\mu$.  If $p =x_i\in T$ and $p' \notin T$, then
    $i \leq i(p')$, hence, $p < p'$ in $K$.  If $p \notin T$ and
    $p' =x_i\in T$, then $i(p) < i$ as otherwise
    $p' \leq x_{i(p)} < p$ in $P$, which is a contradiction.  From
    $i(p) < i$ we then get $p < p'$ in $K$.  Finally, assume that
    $p,p' \notin T$. By transitivity $i(p) \leq i(p')$ follows from
    $p < p'$.  If $i(p) = i(p')$, then since $K_{i(p)}$ is a linear
    extension of the subposet of $P$ induced by the elements of
    $X_{i(p)}$, $p < p'$ in $K$.  If $i(p) < i(p')$, then clearly
    $p < p'$ in $K$.  This completes the proof that $K$ is a linear
    extension of~$P$.  The proof that $L$ is a linear extension of $P$
    is symmetric.}
\end{proof}

\begin{corollary}\label{cor:dim-2at}
    For every poset $P$, $\dim(P)\le 2\at(P)$.
\end{corollary}
\begin{proof}
  Let $P$ be a poset.  By~\Cref{lem:extensions-on-coordinates}, there
  exists a collection~$\cal P$ of maps for~$P$ of size $\at(P)$ such
  that every element of $P$ is tight on at least one of the maps, and
  for every map the two linear orders of the elements of $P$ given by
  the orders on the coordinates of the map are linear extensions of
  $P$.  We claim that the collection $\cal L$ of $2\at(P)$ mentioned
  linear extensions of $P$ is a realizer of $P$.

  Let $x$ and $y$ be two incomparable elements in $P$.  To conclude,
  we show that there is a linear extension in $\cal L$ such that
  $x < y$ in $L$.  Let $\mu \in \cal P$ be a map on which $y$ is
  tight.  In particular, $\mu(y) \not\leq \mu(x)$, and so, there is an
  $i \in [2]$ such that $\mu(x)_i < \mu(y)_i$.  This completes the
  proof.
\end{proof}

\begin{lemma}
  \label{lem:two-chains}
  For every poset $P$ and for every two chains $C$ and~$D$ in~$P$,
  there exists a map~$\mu$ for~$P$ such that all the elements of $C$
  and $D$ are tight on $\mu$.
\end{lemma}
\begin{proof}
  Let $P$ be a poset, and let $C = c_1 <\dots <c_k$ and
  $D = d_1<\dots <d_\ell$ be chains in $P$.  Similarly as in the proof
  of~\Cref{lem:extensions-on-coordinates}, for each element $p$ of
  $P$, we define
  \[i(p) = \max(\{i \in [k] : c_i \leq p\} \cup \{0\}) \ \ \text{ and } \ \ 
    j(p) = \max(\{j \in [\ell] : d_j \leq p\} \cup \{0\}).\]
  Next, for each $i \in [k] \cup\{0\}$ and $j \in [\ell] \cup \{0\}$, let
  \[C_i = \{p \in P\setminus C : i(p) = i\} \ \ \text{ and } \ \ D_j =
    \{p \in P\setminus D : j(p) = j\}.\] For each
  $i \in [k] \cup \{0\}$, let $K_i$ be an arbitrary linear extension
  of the subposet of $P$ induced by the elements in $C_i$, and for
  each $j \in [\ell]\cup \{0\}$, let $L_j$ be an arbitrary linear
  extension of the subposet of $P$ induced by the elements in
  $D_j$. Finally, let $\mu$ be defined so that the linear order given
  by $\mu(\cdot)_1$ is the concatenation
  $K = K_0c_1K_1c_2K_2\dots K_{k-1}c_kK_k$, and the linear order given
  by $\mu(\cdot)_2$ is the concatenation
  $L = L_0d_1L_1d_2L_2\dots L_{\ell-1}d_\ell L_\ell$.
    
  To complete the proof, we need to verify that the elements of
  $C \cup D$ are tight on $\mu$.  The orders $K$ and $L$ are linear
  extensions of~$P$ by definition.  The elements of $C$ are tight in
  $K$ and the elements of $D$ are tight in $L$. Since the order
  induced by the plane map~$\mu$ is the intersection of $K$ and $L$,
  all elements of $C\cup D$ are tight in $\mu$.
    
  \textcolor{gray}{More formally, let $i \in [k]$ and we will show
    that $\suc_P[c_i] = \{p \in P : \mu(c_i) \leq \mu(p)\}$.  The
    proof for the elements of $D$ is symmetric.  First, let
    $p \in \suc_P[c_i]$.  We claim that $\mu(c_i) \leq \mu(p)$.  This
    is clear when $p = c_i$, hence, assume otherwise.  Since
    $c_i < p$ in $P$ and $K$ is a linear extension of~$P$,
    we have $c_i < p$ in~$K$, too, and hence $\mu(c_i)_1 < \mu(p)_1$.
    Similarly, since $L$ is a linear extension of~$P$,
    we obtain that $\mu(c_i)_2 \le \mu(p)_2$.
    Thus, $\mu(c_i) < \mu(p)$, as desired.  Next, let $p'$ be an element
    of~$P$ with $\mu(c_i) \leq \mu(p')$.  In particular,
    $\mu(c_i)_1 \leq \mu(p')_1$ and $c_i < p'$ in $K$.  Since $C$ is a
    chain, this implies $c_i < p'$ in~$P$ and $p' \in \suc_P[c_i]$.
    This completes the proof.}
\end{proof}

\thmdimlepw*
\label{thm:dim-le-2pw*}
\vspace*{-1ex}
\begin{proof}
  \Cref{cor:dim-2at} implies the first inequality.  By Dilworth's
  theorem, a poset $P$ can be covered by a family of $\width(P)$
  chains, hence,~\Cref{lem:two-chains} implies the second inequality
  of the statement.
\end{proof}

\section{The Constructions}
\label{sec:lower}

In this section, we prove \cref{thm:dim-eq-2at,thm:pw-2-2-dim}. In the
core of the arguments, we use the following 6-element poset $Q$
defined on the elements $\{a,b,c,ab,bc,ac\}$ where $a,b,c$ are minimal
elements, $ab,bc,ac$ are maximal elements, and the order relations
are: $a < ab$, $a <ac$, $b < ab$, $b < bc$, $c < ac$, $c < bc$; see
\cref{fig:intro}(a).  Among order theorists, the poset~$Q$ is known as
a \emph{crown} or \emph{cycle} and also as the \emph{standard
  example}~$S_3$.  It is well known that the dimension of $Q$ is~$3$.
This fact clearly implies~\Cref{lem:posetQ} below.  For completeness,
we give a direct proof of the lemma anyway.

\begin{lemma}
  \label{lem:posetQ}
  There does not exist a map of $Q$ where all of $a,b,c$ are tight.
\end{lemma}
\begin{proof}
  \textcolor{gray}{Suppose to the contrary that $\mu$ is a map of $Q$
    where all of $a,b,c$ are tight. Since they are all tight and they
    are pairwise incomparable in $Q$, we may assume that
    $\mu(a)_1 < \mu(b)_1 < \mu(c)_1$ and
    $\mu(c)_2 < \mu(b)_2 < \mu(a)_2$.  Since $a < ac$ and $c < ac$ in
    $Q$, we have $\mu(c)_1 < \mu(ac)_1$ and $\mu(a)_2 < \mu(ac)_2$.
    It follows that $\mu(b)_1 < \mu(ac)_1$ and $\mu(b)_2 < \mu(ac)_2$,
    hence, $\mu(b) < \mu(ac)$.  However, $b$ and $ac$ are incomparable
    in $Q$, which is a contradiction with $b$ being tight in the
    considered map; see \cref{fig:intro}(b), where $b$ is not tight.}
\end{proof}

\thmdimeqat*
\label{thm:dim-eq-2at*}
\vspace*{-1ex}
\begin{proof}
  For every positive integer $m$, let $S_m$ be the poset consisting of
  minimal elements $a_1,\dots,a_{m}$ and maximal elements
  $b_1,\dots,b_{m}$ so that for all $i,j \in [m]$, we have $a_i < b_j$
  if and only if $i \neq j$.  In other words, $S_m$ is the
  \emph{standard example} of order $m$.  Standard examples are the
  canonical posets forcing dimension to be large, and they date back
  to the foundational paper of Dushnik and Miller~\cite{DM41} that
  initiated the study of dimension of posets.  It is well-known that
  the dimension of the standard example of order $m$ is exactly $m$.

  We show that the poset $P_n = S_{2n}$ witnesses the statement for
  every positive integer~$n$.  We have $\dim(S_{2n}) = 2n$.  By
  \cref{thm:dim-le-2pw}, $\dim(S_{2n}) \leq 2\at(S_{2n})$, which
  implies that $\at(S_{2n}) \ge n$.  Hence, it suffices to show that
  $\at(S_{2n}) \leq n$.  For every $i \in [n]$, there exists a map
  where $a_{2i-1}$ and $a_{2i}$ are tight; see \cref{fig:a2ia2i+1}.
  Moreover, for every $j \in [2n]$, $b_j$ is tight in each of these
  maps.  This proves the desired inequality.
\end{proof}

\begin{figure}[tb]
  \centering
  \includegraphics[page=8]{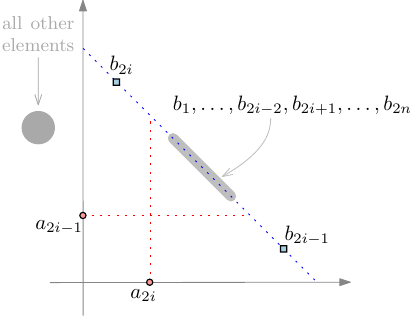}
  \caption{Illustration for the proof of \cref{thm:dim-eq-2at}: A map
    where $a_{2i-1}$ and $a_{2i}$ are tight.}
  \label{fig:a2ia2i+1}
\end{figure} 

\thmpwgedim*
\label{thm:pw-2-2-dim*}
\vspace*{-1ex}
\begin{proof}
  For every positive integer $m$, let $P_m$ be the incidence poset of
  the complete graph on $m$ vertices, $K_m$.  The dimension of such
  posets is well understood.  In fact, in the 1970s, Spencer proved
  that $\dim(P_m) = \mathcal{O}(\log \log m)$.  Later, Ho\c{s}ten and
  Morris~\cite{HOSTEN1999133} and West~\cite[p.~580]{W21} gave more
  precise estimates of $\dim(P_m)$.  To complete the proof, we show
  that for every positive integer~$m$, $\at(P_{2m}) \geq m$.

  Let $m$ be a positive integer.  We claim that there is no triple
  $u,v,w$ of the vertices of $K_{2m}$ such that $u,v,w$ are all tight
  on a single map of $P_{2m}$.  Indeed, the subposet of $P_{2m}$
  formed by $u,v,w,uv,vw,uw$ is isomorphic to the poset~$Q$.  Hence,
  the claim follows from~\Cref{lem:posetQ}.  This implies that
  $m \le \at(P_{2m})$, as desired.
\end{proof}

\section{Complexity}
\label{sec:hardness}
\label{se:hard}

In this section, we prove that it is \NP-complete to compute the atlas
thickness or the dominance mapability of a given poset
(\Cref{thm:hardness-at,thm:hardness-dmap}).  The main idea is to
consider incidence posets of graphs and to show that a subset of tight
vertices in such a poset corresponds to a \emph{linear forest}, that
is, a collection of induced paths, in the graph
(see~\Cref{le:linearforest_to_dom,le:dom_to_linearforest}).

In the case of the atlas thickness of a poset
(\Cref{thm:hardness-at}), we will reduce from the problem of
partitioning the vertex set of a graph into two sets such that each
induces a linear forest.  The \NP-hardness of this problem was
observed by Chaplick, Fleszar, Lipp, Ravsky, Verbitsky, and
Wolff~\cite[Theorem~12]{cflrvw-cdgfl-JGAA23}.  It follows from a
general result of Farrugia~\cite{f-vpfai-EJC04}, who showed that, for
any additive induced-hereditary properties~$\Pi$, it is \NP-hard to
decide whether the vertex set of a graph can be partitioned into two
sets that have property~$\Pi$.

In the case of the dominance mapability of a poset
(\Cref{thm:hardness-dmap}), we will reduce from the problem of finding
a subset of $k$ vertices of a given graph that induces a linear
forest.  The \NP-hardness of this problem follows from another general
result proved by Lewis and Yannakakis~\cite{LEWIS1980219}, who showed
that removing the minimum number of vertices from a graph~$G$ such
that the remaining induced subgraph has a property~$\Pi$ is \NP-hard
if~$\Pi$ is non-trivial and hereditary.

The crucial ingredient for the proof is the well-known
characterization of incidence posets of dimension at most~$2$, which
is stated in the following proposition.  We provide a direct proof,
which will help us to show the next two lemmas.

\begin{proposition}
  \label{prop:Gdim2}
  The dimension of the incidence poset $P_G$ of a graph $G$ is at
  most~$2$ if and only if $G$ is a linear forest.
\end{proposition}

\begin{proof}
  The incidence order of a linear forest is an induced subposet of the
  incidence poset of a path, hence, it is sufficient to show that the
  incidence poset of a path is 2-dimensional. Let~$G$ be a path with
  vertices $a_1,\ldots,a_{\ell+1}$ and edges $e_i=a_ia_{i+1}$ for
  $i=1,\ldots,\ell$, then the following two linear extensions of $P_G$
  are a realizer of size two:
  $K = a_1,a_2,e_1,a_3,e_2,\ldots,a_{\ell+1},e_\ell$ and
  $L = a_{\ell+1},a_\ell,e_{\ell},\dots, a_2,e_2,a_1,e_1$.

  If $G$ is not a forest, then $G$ contains a cycle $C$. It is well
  known that the dimension of $P_C$ is~$3$, hence $\dim(P_G) \geq 3$
  in this case.  If $G$ is a forest but not a linear forest, then $G$
  contains a 3-star.  The incidence poset of a 3-star is a poset on 7
  elements called \emph{spider}.  It is well known that the dimension
  of the spider is~$3$, hence again $\dim(P_G) \geq 3$.  Cycle posets
  and the spider are indeed 3-irreducible posets, i.e., they are
  3-dimensional and all their subposets are 2-dimensional,
  cf.~\cite[p.~573]{W21} and the references given there.
\end{proof}

\begin{lemma}
  \label{le:linearforest_to_dom}
  Let $G$ be a graph and let $U \subseteq V(G)$.  If $G[U]$ is a
  linear forest, then there exists a map of $P_G$ such that the
  elements in $U \cup E(G)$ are tight.
\end{lemma}

\begin{proof}
  In the proof of \cref{prop:Gdim2}, we have shown how to construct
  two linear orders~$K$ and~$L$ that form a realizer of~$P_{G[U]}$.
  Edges of~$G$ that are incident to just one element~$u$ in~$U$ can be
  placed in~$K$ and~$L$ just above~$u$.  All the remaining elements
  of~$P_G$ can be put at the beginning of $K$ and at the end of~$L$.
  \Cref{fig:hard-2} illustrates the construction.
\end{proof}

\begin{figure}[h]
  \begin{minipage}[b]{.6\textwidth} 
    \begin{subfigure}{\textwidth}
        \centering
        \includegraphics[page=2]{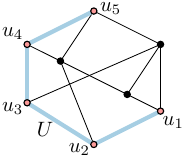}
        \subcaption{A graph $G$ with a set $U$ of vertices such that
          $G[U]$ is a path.}
        \label{fig:hard-2-a}
    \end{subfigure}\\

    \begin{subfigure}{\textwidth}
        \centering
        \includegraphics[page=4]{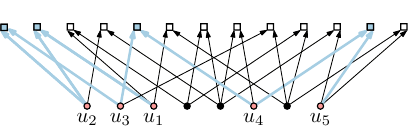}
        \subcaption{The incidence poset $P_G$ of~$G$.}
        \label{fig:hard-2-b}
    \end{subfigure}
  \end{minipage}
  \hfill
  \begin{minipage}[b]{.38\textwidth}      
    \centering
    \includegraphics[page=5]{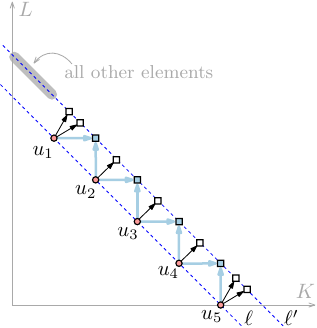}
    \subcaption{The map~$\mu$.}
    \label{fig:hard-2-c}
  \end{minipage}
  \caption{Illustration for the proof of
    \cref{le:linearforest_to_dom}.  The map~$\mu$ places the vertices
    of~$U=\{u_1,\dots,u_5\}$ (light red disks) on the line~$\ell$ in
    the same order as along the path~$G[U]$.  Each edge (light blue
    square) that connects two vertices in~$U$, say, $u_i$
    and~$u_{i+1}$, is placed on~$\ell'$ to the left of~$u_i$ and
    above~$u_{i+1}$.  Each edge (white square) incident to exactly one
    vertex~$u_i$ in~$U$ is placed on~$\ell'$ to the right and
    above~$u_i$.  Finally, all other elements are placed on the side
    (here at the top of $\ell'$).}
    \label{fig:hard-2}
\end{figure} 

\begin{lemma}
  \label{le:dom_to_linearforest}
  Let $G$ be a graph, let $U \subseteq V(G)$, and let $\mu$ be a map
  of $P_G$.  If all the elements of $U$ are tight on $\mu$, then
  $G[U]$ is a linear forest.
\end{lemma}

\begin{proof}
  This follows directly from \Cref{prop:Gdim2}.
\end{proof}

\thmlva*
\label{thm:lva*}
\vspace*{-1ex}
\begin{proof}
  By \cref{le:linearforest_to_dom}, $\at(P_G) \le \lva(G)$.  On the
  other hand, by \cref{le:dom_to_linearforest},
  $\lva(G) \le \at(P_G)$.  Hence, $\lva(G)=\at(P_G)$.
\end{proof}

\corhardnessat*
\label{thm:hardness-at*}
\vspace*{-1ex}
\begin{proof}
  The problem is clearly in \NP\ since we can verify a solution in
  polynomial time.

  It is \NP-hard to test whether a given graph has linear vertex
  arboricity at most~2~\cite[Theorem~12]{cflrvw-cdgfl-JGAA23}.  Hence
  the same holds for testing the atlas thickness of a(n incidence)
  poset.
  
  For the non-approximability, observe that the linear vertex
  arboricity and the chromatic number of a graph are linearly
  related~\cite[Theorem~12]{cflrvw-cdgfl-JGAA23}.  More precisely, for
  every graph~$G$, we have $\chi(G) \le \lva(G) \le 2 \cdot \chi(G)$
  since (i)~in every coloring, each color class is a linear forest and
  (ii)~every linear forest can be 2-colored.  Now the theorem follows
  from the following two facts: (a)~it is \NP-hard to approximate, for
  a graph~$G$, $\chi(G)$ within a factor of
  $O(|V(G)|^{1-\varepsilon})$~\cite{z-ldeim-TC07} and (b)~the size of
  the incidence poset of $G$ is at most quadratic in the number of
  vertices of~$G$, that is, $|P_G| \in O(|V(G)|^2)$.
\end{proof}

\thmhardnessdmap*
\label{thm:hardness-dmap*}
\vspace*{-1ex}
\begin{proof}
  The problem is clearly in \NP\ since we can verify a solution in
  polynomial time.

  We reduce from the problem of finding, given a graph and a positive
  integer~$k$, a subset of at least $k$ vertices that induces a linear
  forest.  The problem is \NP-hard due to a general result of Lewis
  and Yannakakis~\cite{LEWIS1980219}; see the discussion at the
  beginning of this section.  We show that a graph $G$ contains a
  subset of at most $k$ vertices that induces a linear forest in $G$
  if and only if $\dmap(P_G) \ge k + |E(G)|$.

  Given such a subset, \Cref{le:linearforest_to_dom} immediately
  implies that there exists a map of $P_G$ where at least $k+|E(G)|$
  elements are tight, and so, $\dmap(P_G) \ge k + |E(G)|$.

  On the other hand, suppose that we are given a map of~$P_G$ where at
  least $k+|E(G)|$ elements are tight.  By the pigeonhole principle,
  at least $k$ of these elements are in $V(G)$.  It follows
  by~\cref{le:dom_to_linearforest} that this subset of vertices
  induces a linear forest in $G$, as desired.
\end{proof}

\section{An FPT-Algorithm for Dominance Mapability}
\label{se:FPT}

In this section, we prove~\cref{thm:size-dmap}, which implies that,
given a poset $P$, computing its mapability $\dmap(P)$ is \FPT\ with
respect to the natural parameter (see \cref{cor:FPT}).  
Namely, we
will upperbound the size of a poset in terms of its mapability.  The
next statement is a direct corollary of~\Cref{lem:two-chains} applied
to a maximum chain and a singleton.%

\begin{corollary}
  \label{le:height}
  For every poset $P$ of height~$h$ that is not a chain, we have
  $h \le \dmap(P)-1$.
\end{corollary}

A subset $M$ of $P$ is an \emph{up-module} if any two elements of $M$
have the same successors, i.e., $\suc_P(u) = \suc_P(v)$ for all
$u,v\in M$.  Note that an up-module must be an antichain.  If $M$ is
an up-module, then we write \emph{$\suc_P(M)$} to denote the set of
successors of the elements of $M$.

\begin{lemma}
  \label{le:box}
  Let $P$ be a poset.  If $M$ is an up-module of $P$, then
  $|M|\le \dmap(P)$.
\end{lemma}
\begin{proof}
  Let $\mu$ be a map where some element $v\in M$ is tight.  By
  Lemma~\ref{lem:extensions-on-coordinates} we can assume that the
  coordinate projections of $\mu$ are linear extensions of $P$.  We
  can also assume that all coordinates are integers. Let $\ell$ be the
  line of slope $-1$ through $\mu(v)$ and let $s$ be a segment of
  length 1 on $\ell$ which contains $\mu(v)$. Note that our
  assumptions about $\mu$ imply that
  $\suc_P(M)=\suc_P(v) = \{ x \in P : q <_\mathrm{dom} \mu(x) \}$ for
  every point $q$ on $s$.  Hence, if we place the elements of~$M$ at
  distinct points of~$s$, we create a map where all elements of $M$ are
  tight.
\end{proof}

Observe that the set $\Max(P)$ is an up-module.
Hence, \cref{le:box} implies the following statement.

\begin{corollary}
    \label{co:sinks}
    For every poset $P$, we have $|\Max(P)| \le \dmap(P)$.
\end{corollary}

Given two positive integers $a$ and $k$, ${}^{k}a$ is a tower of $a$'s
of height $k$.  For example, $^{3}2=2^{2^2}=16$. More formally,
${}^1a=a$ and, for any positive integer $i$,
${}^i a=a^{({}^{i-1} a)}$. We extend this notation and write
${}^{k}_{b}a$ to denote a tower of total height $k$ where the top
element is $b$ and the remaining ($k-1$) elements are $a$'s.  For
example, $^{3}_42=2^{2^4}=2^{16}$.

We prove the following theorem with respect to the function
$f(d) = {}^{d}_d4$.

\thmsizetodmap*
\label{thm:size-dmap*}
\vspace*{-1ex}
\begin{proof}
  Let $h$ be the height of $P$, and let $d = \dmap(P)$.  The
  \emph{canonical antichain partition} of~$P$ is the partition of the
  elements of $P$ into $h$ antichains.  Deviating from the standard
  definition we construct the canonical antichain partition top-down
  instead of bottom-up.  Let $A_{1}=\Max(P)$ and, for each
  $k\in [h]\setminus\{1\}$, let $A_k$ be the set of maximal elements
  of the induced poset $P \setminus \bigcup_{j=1}^{k-1} A_j$.  We will
  identify a function~$f$ such that, for every $k \in [h]$,
  $|A_k| \le f(k)$.  The function~$f$ will grow very fast.  In
  particular, for every $k \ge 2$, we will have
  \begin{equation}
    \label{eqn:growth}
    f(k) \ge 3\sum_{j=1}^{k-1} f(j) \quad\text{and}\quad
    d \le 2^{(2/3)f(k-1)}. \tag{$\star$}
  \end{equation}
  Since $A_1 = \Max(P)$, we obtain from~\cref{co:sinks} that
  $|A_1| \le d$.  We set $f(1) = d$.  Note that this yields
  $|A_1| \le f(1)$ and fulfills the right inequality
  in~\eqref{eqn:growth} for every $k \ge 2$.  (For $k=2$, this is due
  to the fact that $x \le 2^{{2}x/{3}}$ holds for every $x>0$.  For
  $k>2$, this is due to the fact that $f$ grows.)

  Now consider $k \in [h] \setminus \{1\}$, and let $\cal M$ be a
  partition of $A_k$ into maximal up-modules.  By~\cref{le:box}, for
  every $M \in \cal M$, we have $|M| \le d$.  In particular,
  $|A_k| \le d \cdot |\cal M|$. A~module $M \in \cal M$ is determined
  by $\suc_P(M)$, which is a subset of $\bigcup_{j=1}^{k-1} A_j$.  Due
  to \eqref{eqn:growth}, i.e., the fast growth of~$f$, we have
  \[
    \left|\bigcup_{j=1}^{k-1} A_j\right| \le \sum_{j=1}^{k-1} f(j)
    = f(k-1) + \sum_{j=1}^{k-2} f(j) \le \frac{4}{3}f(k-1).
  \]
  This implies that there are at most $2^{(4/3)f(k-1)}$ choices for
  $\suc_P(M)$.  Again using \eqref{eqn:growth}, we get
  \[
    |A_k| \le d \cdot 2^{(4/3)f(k-1)} \leq 2^{(2/3)f(k-1)}2^{(4/3)f(k-1)}
    = 2^{2f(k-1)} = 4^{f(k-1)}.
  \]
  If we now set $f(k)= {}^k_d4$, then \eqref{eqn:growth} is fulfilled
  for $k \ge 2$, and $|A_k| \le f(k)$ for $k \in [h]$ (including
  $k=1$).  Using \eqref{eqn:growth}, \cref{le:height} (which says that
  $h+1 \le d$ if $P$ is not a chain), and the fact that $f$ is
  increasing, we finally get
  \[
    |P| \le \sum_{j=1}^{h} f(j) \le \frac{1}{3}f(h+1) \le f(d) = {}^d_d4. 
  \]%
\end{proof}

\section{Open Problems and Future Directions}

We have introduced the new concepts of mapability and atlas thickness,
and we conclude by highlighting several open problems for future
investigation.

\begin{enumerate}
\item If we embed into $\mathbb{R}$ instead of $\mathbb{R}^2$, we just
  get the width of the poset. Is it interesting to consider higher
  dimensions?

\item We have proved that, for every poset $P$,
  $\dim(P) \leq 2 \at(P)$ and that there exist posets (incidence
  posets of complete graphs) with the atlas thickness huge relative to
  their dimension.  However, we did not find a family of posets with
  constant dimension and unbounded atlas thickness.  On the other
  hand, it is well-known that the Boolean dimension of incidence
  posets of complete graphs is bounded, thus our family of posets
  \textit{separates} Boolean dimension and atlas thickness.  It would
  be interesting to understand if dimension and atlas thickness are
  functionally equivalent or not.

\item The following class of parameters of posets may be worth
  studying: For a given poset~$P$ and positive integer~$k$, find the
  minimum $t(k)$ such that there exists a family $\cal R$ of linear
  extensions with:
  \begin{itemize}
  \item the size of $\cal R$ is at most $k\cdot\dim(P)$,
  \item for each $v$ in $P$, there is a subfamily
    $R(v) \subseteq \cal R$ of size at most $t(k)$ such that
    $\suc_P[v] = \{w\in P \colon v \leq_L w$ for every $L \in R(v) \}$
    (we could say that $v$ is tight in~$R(v)$).
  \end{itemize}
  The motivation for the definition is that we obtain a data structure
  that allows us to answer comparison queries ($<$, $\parallel$, $>$)
  with space requirement $O(k \cdot \dim(P) \cdot n)$ and query time
  $O(t(k))$.
\end{enumerate}

\bibliographystyle{plainurl}
\bibliography{dominance}

\end{document}